\documentclass[letterpaper, 10 pt, conference]{ieeeconf}
\IEEEoverridecommandlockouts 
\usepackage{flushend}

\usepackage{amssymb}
\usepackage[cmex10]{amsmath}
\usepackage{algorithm}
\usepackage{algpseudocode}
\usepackage{graphicx}
\usepackage{color}
\usepackage{subcaption}
\usepackage{blindtext}
\usepackage{dsfont}
\usepackage{booktabs}

\algrenewcommand\algorithmicindent{2mm}%

\graphicspath{{./img/}}

\usepackage{cancel}

\newtheorem{theorem}{Theorem}

\newtheorem{assumption}{Assumption}

\newcommand{\xta}[1]{\tilde{x}^{(#1)}}

\newcommand{\xt}[1]{\tilde{x}^{(#1)}}
\newcommand{\xs}[1]{x^{(s_{#1})}}
\newcommand{\xsh}[1]{\hat{x}^{(s_{#1})}}

\newcommand{\xopt}{x^{\star}}
\newcommand{\N}[1]{\mathcal{N}_{#1}}
\newcommand{\htil}[1]{\tilde{h}^{(#1)}}
\newcommand{\lt}{\ell}

\newcommand{\norm}[1]{\|#1\|}
\newcommand{\abs}[1]{|#1|}
\newcommand{\del}{\nabla}

\newcommand{\eqdef}{\ensuremath{\triangleq}}

\newcommand{\A}{\ensuremath{\mathcal{A}}}
\newcommand{\B}{\ensuremath{\mathcal{B}}}

\newcommand{\expec}[1]{\textbf{E}\left[#1\right]}
\newcommand{\ex}{\textbf{E}}
\newcommand{\mR}{\mathbb{R}}
\newcommand{\tp}{\ensuremath{^{\mathsf{T}}}}
\newcommand{\xc}[1]{\ensuremath{x}^{(#1)}}

\newcommand{\ec}[1]{\ensuremath{e^{(#1)}}}
\newcommand{\zbar}{\overline{z}}
\newcommand{\qerr}[1]{\ensuremath{\varepsilon(#1)}}
\newcommand{\prox}[1]{\ensuremath{\text{prox}_{\eta R}(#1)}}
\newcommand{\vc}[1]{\ensuremath{v}^{(#1)}}
\newcommand{\cA}[1]{\ensuremath{\mathcal{A}_{#1}}}
\newcommand{\xth}[2]{\hat{\tilde{x}}^{(#1)}_{#2}}

\newcommand{\cN}[1]{\ensuremath{\mathcal{N}_{#1}}}

\newcommand{\ani}[1]{a_{i}^{(#1)}}

\newcommand{\bni}[1]{b_{i}^{(#1)}}
\newcommand{\cni}[1]{c_{i}^{(#1)}}

\newcommand{\dni}[1]{d_{i}^{(#1)}}
\newcommand{\Uas}[1]{U_{a,i}^{(#1)}}
\newcommand{\Ubs}[1]{U_{b,i}^{(#1)}}
\newcommand{\Ucs}[1]{U_{c,i}^{(#1)}}
\newcommand{\Uds}[1]{U_{d,i}^{(#1)}}
\newcommand{\Lmax}{\overline{L}}
\newcommand{\qgradi}[1]{\hat{\del} f_i^{(#1)}}
\newcommand{\gradi}[1]{{\del} f_i^{(#1)}}

\newcommand{\qgrad}[2]{\hat{\del} f_{#1}^{(#2)}}
\newcommand{\gc}[1]{\ensuremath{\tilde{g}}^{(#1)}}
\newcommand{\wc}[1]{\ensuremath{w}^{(#1)}}
\newcommand{\xplus}{x^{+}}
\newcommand{\hc}[1]{\ensuremath{h}^{(#1)}}

\newcommand{\amid}{\overline{x}_{a,i}^{(s)}}
\newcommand{\bmid}{\overline{\del}f_{b,i}^{(s)}}
\newcommand{\cmid}{\overline{x}_{c,i}^{(s)}}
\newcommand{\dmid}{\overline{\del}f_{d,i}^{(s)}}

\newcommand{\Qai}{Q_{a,i}^{(s)}}
\newcommand{\Qbi}{Q_{b,i}^{(s)}}
\newcommand{\Qci}{Q_{c,i}^{(s_t)}}
\newcommand{\Qdi}{Q_{d,i}^{(s_t)}}

\DeclareMathOperator*{\argmin}{arg\,min}

\newtheorem{lemma}{Lemma}

\title{\LARGE \bf Distributed Semi-Stochastic
Optimization with Quantization Refinement}

\author{Neil McGlohon and Stacy Patterson
\thanks{*This work was funded in part by NSF grants 1553340 and 1527287.}
\thanks{N. McGlohon and S. Patterson are with the Department of Computer Science, Rensselaer Polytechnic Institute,
Troy, NY 12180, USA
       {\tt\small mcglon@rpi.edu}, {\tt\small sep@cs.rpi.edu}}%
}

\begin{document}

\maketitle
\thispagestyle{empty}
\pagestyle{empty}

\begin{abstract}
We consider the problem of regularized regression in a network of communication-constrained devices. 
Each node has local data and objectives, and the goal is for the nodes to optimize a global objective.
We develop a distributed optimization algorithm that is based on recent work on semi-stochastic proximal gradient methods.
Our algorithm employs iteratively refined quantization to limit message size.
We present theoretical analysis and conditions for the algorithm to achieve a linear convergence rate.  Finally, we demonstrate the performance of our algorithm through numerical simulations.
\end{abstract}

\section{Introduction}

We consider the problem of distributed optimization in a network where communication is constrained, for example a wireless sensor network.
In particular, we focus on problems where each node has local data and objectives, and the goal is for the nodes to learn a global objective that includes this local information.
Such problems arise in networked systems problems such as estimation, prediction, resource allocation, and control.

Recent works have proposed distributed optimization methods that reduce communication by using quantization. 
For example, in \cite{PY2015a}, the authors propose a distributed algorithm to solve unconstrained problems based on a centralized inexact proximal gradient method~\cite{SRB11}.
In~\cite{PY2015b}, the authors extend their work to constrained optimization problems. 
In these algorithms, the nodes compute a full gradient step in each iteration, requiring quantized communication between every pair of neighboring nodes.
Quantization has been applied in distributed consensus algorithms~\cite{AK2007,DT2013,RC2007}
and distributed subgradient methods~\cite{AN2008}.

In this work, we address the specific problem of distributed regression with regularization over the variables across all nodes.  Applications of our approach include
distributed compressed sensing, LASSO, group LASSO, and regression with Elastic Net regularization, among others.
Our approach is inspired by~\cite{PY2015a,PY2015b}.
We seek to further reduce per-iteration communication by using an approach based on a stochastic proximal gradient algorithm.
This approach only requires communication between a small subset of nodes in each iteration.  In general, stochastic gradients may suffer from slow convergence.
Thus any  per-iteration communication savings could be counter-acted by an extended number of iterations.  Recently, however, several works have proposed \emph{semi-stochastic}
gradient methods~\cite{NA2014,JR2013,xiao2014proximal}.  To reduce the variance of the iterates generated by a stochastic approach, these algorithms periodically incorporate a full gradient computation.  It has been shown that these algorithms achieve a linear rate of convergence to the optimal solution.

We propose a distributed algorithm for regularized regression based on the centralized semi-stochastic proximal gradient of~\cite{xiao2014proximal}. 
In most iterations, only a subset of nodes need communicate.
We further reduce communication overhead by employing quantized messaging.   Our approach reduces both the length of messages sent between nodes as well as the number of messages sent in total to converge to the optimal solution.
The detailed contributions of our work are as follows:
\begin{itemize}
\item We extend the centralized semi-stochastic proximal gradient algorithm to include errors in the gradient computations and show the convergence rate of this inexact algorithm.
\item We propose a distributed optimization algorithm based on this centralized algorithm that uses iteratively refined quantization to limit message size.
\item We show that our distributed algorithm is equivalent to the centralized algorithm, where the errors introduced by quantization can be interpreted as inexact gradient computations. We further design quantizers that guarantees a linear convergence rate to the optimal solution.
\item We demonstrate the performance of the proposed algorithm in numerical simulations.
\end{itemize} 

The remainder of this paper is organized as follows.  
In Section~\ref{prelim.sec}, we present the centralized inexact proximal gradient algorithm and give background on quantization. 
In Section~\ref{problem.sec}, we give the system model and problem formulation.
Section~\ref{algorithm.sec} details our distributed algorithm.
Section~\ref{analysis.sec} provides theoretical analysis of our proposed algorithm.
Section~\ref{results.sec} presents our simulation results, and we conclude in Section~\ref{conclusion.sec}.

\section{Preliminaries}~\label{prelim.sec}

\subsection{Inexact Semi-Stochastic Proximal Gradient Algorithm}~\label{SSPG.sec}
We consider an optimization problem over the form: 
\begin{equation} \label{centralopt.eq}
 \underset{x \in \mR^P}{\text{minimize}}~~G(x) = F(x) + R(x),  
\end{equation}
where $F(x) = \frac{1}{N} \sum^N_{i=1} f_i(x)$, and the following assumptions are satisfied. 
\begin{assumption} \label{fi.assum}
Each $f_i(x)$ is differentiable, and its gradient $\nabla f_i(x)$ is Lipschitz continuous with constant $L_i$, i.e., for all $x,y \in \mR^P$,
\begin{equation} \label{lipschitz.eq}
\| \nabla f_i (x) - \nabla f_i(y) \| \leq L_i \| x - y \|.
\end{equation}
\end{assumption}

\begin{assumption} \label{R.assum}
The function $R(x)$ is lower semicontinuous, convex, and its effective domain, $\text{dom}(R) := \{x \in \mR^P ~|~ R(x) < +\infty\}$,  is closed.
\end{assumption}

\begin{assumption} \label{strongconvex.assum}
The function $G(x)$ is strongly convex with parameter $\mu > 0$, i.e., for all $x, y \in \text{dom}(R)$
and for all $\xi \in \partial G(x)$,
\begin{equation} \label{strongcvx.eq}
G(x) - G(y) - \textstyle \frac{1}{2} \mu \|x - y \|^2 \geq \xi \tp (x - y),
\end{equation}
where $\partial G(x)$ is the subdifferential of $G$ at $x$. This strong convexity may come from either $F(x)$ or $R(x)$ (or both).
\end{assumption}


Problem (\ref{centralopt.eq}) can be solved using a stochastic proximal gradient algorithm~\cite{DS09} where, in each iteration,  
a single $\nabla f_{\ell}$ is computed for a randomly chosen $\ell \in \{1, \ldots, N\}$, and the iterate is updated accordingly as,
\[
\xc{t+1} = \prox{\xc{t} - \eta^{(t)} \nabla f_{\ell}(\xc{t})}.
\]
Here, $\prox{\cdot}$ is the proximal operator 
\[
\prox{v} = \argmin_{y \in \mR^p} \frac{1}{2} \| y- v \|^2  + \eta R(y).
\]

While stochastic methods offer the benefit of reduced per-iteration computation over standard gradient methods,
the iterates may have high variance. These methods typically use a decreasing step-size $\eta^{(t)}$ to compensate for this variance, resulting in slow convergence.
Recently, Xiao and Zhang proposed a semi-stochastic proximal gradient algorithm, {Prox-SVRG} that reduces the variance by periodically incorporating a full gradient computation~\cite{xiao2014proximal}.
This modification allows Prox-SVRG to use a constant step size, and thus, {Prox-SVRG} achieves a linear convergence rate.

\begin{algorithm}[t]
\caption{Inexact Prox-SVRG.} \label{centStochastic.alg}
\begin{algorithmic} \footnotesize
\State \textbf{Initialize:} $\xta{s} = 0$
\For{$s = 0, 1, 2, \ldots$}
\State $\gc{s} =  \nabla F(\xta{s})$
\State $\xc{s_0} = \xta{s}$
\For{$t = 0, 1, 2, \ldots, T-1$}
\State Choose $\ell$ uniformly at random from $\{1, \ldots, N \}$.
\State $\vc{s_t} = \nabla f_{\ell}(\xc{s_{t}}) - \nabla f_{\ell}(\xta{s}) + \gc{s} + \ec{s_t}$
\State $\xc{s_{t+1}} = \prox{\xt{s_{t}} - \eta \vc{s_t}}$
\EndFor
\State $\xta{s+1} =  \frac{1}{T} \sum_{t=1}^T\xta{s_t}$ 
\EndFor
\end{algorithmic}
\end{algorithm}

We extend {Prox-SVRG} to include a zero-mean error in the gradient computation.  Our resulting algorithm, Inexact Prox-SVRG, 
is given in Algorithm~\ref{centStochastic.alg}.
The algorithm consists of an outer loop where the full gradient is computed and an inner loop where the iterate is updated based on both the stochastic and full gradients.

The following theorem states the convergence behavior of Algorithm~\ref{centStochastic.alg}. 
\begin{theorem}\label{centralized.thm}
Let $\{\xt{s}\}_{s \geq 0}$ be the sequence generated by Algorithm~\ref{centStochastic.alg}, with $0 < \eta < \frac{1}{4 \Lmax}$,
where $\Lmax = \max_i L_i$. 
Assume that the functions $R$, $G$, and $f_i$, $i=1, \ldots, N$, satisfy Assumptions~\ref{fi.assum}, \ref{R.assum}, and \ref{strongconvex.assum},
and that the errors $\ec{s_t}$ are zero-mean and uncorrelated with the iterates $\xc{s_t}$ and their gradients $\nabla f_i(\xc{s_t})$.
 Let ${\xopt = \arg \min_x G(x)}$, and let $T$ be such that,
\[
\alpha = \frac{1}{\mu \eta(1 - 4 \Lmax \eta)T} + \frac{4\Lmax \eta(T+1)}{(1 - 4 \Lmax \eta)T} < 1.
\]
Then,
\begin{align*}
&\expec{G(\xta{s}) - G(\xopt)} \\
&~~~~~ \leq \alpha^s \left(G(\xt{0}) - G(x^\star) + \beta \sum^{s}_{i=1} \alpha^{-i} \Gamma^{(i)} \right)
\end{align*}
where $\beta = \frac{\eta}{T(1-4 \Lmax \eta)}$ and $\Gamma^{(i)} = \sum^{T-1}_{t=0} \ex{\norm{\ec{i_t}}^2}$.
\end{theorem}
The proof is given in the appendix.

From this theorem, we can derive conditions for the algorithm to converge to the optimal $\xopt$.
Let the sequence $\{\Gamma^{(s)}\}_{s \geq 0}$ decrease linearly at a rate  $\kappa$. 
Then
\begin{enumerate}
\item If $\kappa < \alpha$, then $\expec{G(\xta{s}) - G(\xopt)}$ converges linearly with a rate of $\alpha$.
\item If  $\alpha < \kappa < 1$, then $\expec{G(\xta{s}) - G(\xopt)}$ converges linearly with a rate of $\kappa$.
\item If  $\kappa = \alpha$, then $\expec{G(\xta{s}) - G(\xopt)}$ converges linearly with a rate in $O(s \alpha^s)$.
\end{enumerate}

\subsection{Subtractively Dithered Quantization}

We employ a subtractively dithered quantizer to quantize values before transmission. 
We use a substractively dithered quantizer rather than non-subtractively dithered quantizer because the quantization 
error of the subtractively dithered quantizer is not correlated with its input.
We briefly summarize the quantizer and its key properties below.

Let $z$ be real number to be quantized into $n$ bits.   
The quantizer is parameterized by an interval size $U$ and a midpoint value $\overline{z} \in \mR$.
Thus the quantization interval is ${[\overline{z} - U/2, \overline{z}  + U/2]}$, and the quantization step-size is  ${\Delta = \frac{U}{2^n - 1}}$.
We first define the uniform quantizer,
\begin{equation} \label{uniquant.eq}
q(z) \eqdef \zbar + \text{sgn}(z - \zbar) \cdot \Delta \cdot \left \lfloor{\frac{\abs{z -\zbar}}{\Delta} + \frac{1}{2}} \right \rfloor.
\end{equation}
In subtractively dithered quantization, a dither $\nu$ is added to $z$, the resulting value is quantized using a uniform quantizer,
and then transmitted.  The recipient then subtracts $\nu$ from this value.
The subtractively dithered quantized value of $z$, denoted $\hat{z}$, is thus
\begin{equation} \label{quantizer.eq}
\hat{z} = Q(z) \eqdef q(z + \nu) - \nu.
\end{equation}
Note that this quantizer requires both the sender and recipient to use the same value for $\nu$, for example, by using the same 
pseudorandom number generator.

The following theorem describes the statistical properties of the quantization error.
\begin{theorem}[See \cite{lipshitz1992quantization}] \label{quantizer.thm}
Let $z \in [\overline{z} - U/2, \overline{z}  + U/2]$ and $\hat{z} = Q(z)$, for $Q(\cdot)$ in (\ref{quantizer.eq}).
Further, let $\nu$ is a  real number drawn uniformly at random from the interval $(-\Delta/2, \Delta/2)$.
The quantization error $\qerr{z} \eqdef z - \hat{z}$ satisfies the following:
\begin{enumerate}
\item $\expec{\qerr{z}} = \expec{\nu} = 0$.
\item $\expec{\qerr{z}^2} = \expec{\nu^2} = \frac{\Delta^2}{12}$ 
\item $\expec{z \qerr{z}} = \expec{z} \expec{\qerr{z}} = 0$ 
\item For  $z_1$ and $z_2$ in the interval ${[\overline{z} - U/2, \overline{z}  + U/2]}$, 
$\expec{\qerr{z_1} \qerr{z_2}} = \expec{\qerr{z_1}}\expec{\qerr{z_2}} = 0.$ \\
\end{enumerate}
\end{theorem}

With some abuse of notation, we  also write $Q(v)$ where $v$ is a vector.
In this case, the quantization operator is applied to each component of $v$ independently, using
a vector-valued midpoint and the same scalar-valued interval bounds.

\section{Problem Formulation} \label{problem.sec}

We consider a similar system model to that in~\cite{PY2015a}.
The network is a  connected graph of $N$ nodes where inter-node communication is limited to the local neighborhood of each node. 
The neighbor set $\cN{i}$ consists of node $i$'s neighbors and itself.
The neighborhoods exist corresponding to the fixed undirected graph $G = (\mathcal{V}, \mathcal{E})$. 
We denote $D$ as the maximum degree of the graph $G$. 

Each node $i$ has a state vector $x_{i}$ with dimension $m_i$.  The state of the system is $x = [ x_1\tp x_2 \tp \ldots x_N \tp]\tp$.
We let $x_{\cN{i}}$ be the vector consisting of the concatenation of states of all nodes in $\cN{i}$.
For ease of exposition, we define the selecting matrices $\A_i$, $i = 1, \ldots, N$, where  $x_{\cN{i}} = \A_i x$ and 
the matrices $\B_{ij}$, $i,j= 1, \ldots, N$ where $x_j = \B_{ij} x_{\cN{i}}$.  These matrices each have  $\ell_2$-norm of 1.

Every node $i$ has a local objective function over the states in $\cN{i}$.
The distributed optimization problem is thus,
\begin{equation} \label{distprob.eq}
\underset{x \in \mR^P}{\text{minimize}}~~G(x) = F(x) + R(x),
\end{equation}
where $F(x) = \frac{1}{N} \sum_{i=1}^N f_i(x_{\cN{i}})$.
We assume that Assumptions~\ref{fi.assum} and \ref{strongconvex.assum} are satisfied.
Further, we require the following assumptions hold.
\begin{assumption} \label{fex.assum}
For all $i$, $\nabla f_i(x_{\cN{i}})$ is linear or constant.  This implies that, for a zero-mean random variable $\nu$, 
$\expec{\nabla f_i(x_{\cN{i}} + \nu)} = \nabla f_i(x_{\cN{i}})$.
\end{assumption}
\begin{assumption} \label{je.assum}
The proximal operation $\text{prox}_R(x)$ can be performed by each node locally, i.e.,
\[
\text{prox}_R(x) = [ \text{prox}_R(x_1)\tp~ \text{prox}_R(x_2)\tp \ldots \text{prox}_R(x_N)\tp]\tp.
\]
\end{assumption}
We note that Assumption~\ref{je.assum} holds for standard regularization functions used in LASSO ($\| x \|_1$), group LASSO where each $x_i$ its own group, 
and Elastic Net regularization ($\lambda_1 \|x \|_1  + \frac{\lambda_2}{2} \| x \|^2_2$).

In the next section, we present our distributed implementation of Prox-SVRG to solve Problem (\ref{distprob.eq}).

%
%

\begin{algorithm}
\caption{Inexact Semi-stochastic Gradient Descent as executed by node $i$} \label{SPGdist.alg}
\begin{algorithmic}[1]  \footnotesize
\State \textbf{Parameters:} inner loop size $T$, step size $\eta$
\State \textbf{Initialize:} $\xt{0}_{i} = 0$, $\xth{-1}{i}=0$, $\qgradi{-1} = 0$
\For{$s=0,1, \ldots$}

\State Update quantizer parameters: 
\Statex $~~~~\Uas{s} = C_a  \kappa^{(s+1)/2}$,  $\amid = \xth{s-1}{i}$, \\
\Statex $~~~~\Ubs{s} = C_b \kappa^{(s+1)/2}$,  $\bmid= \qgradi{s-1}$
\State Quantize local variable and send to all $j \in \cN{i}$:\
\Statex  $~~~~~~~~~~~~~~~\xth{s}{i}=  \Qai(\xt{s}_i) = \xt{s}_i + \ani{s}$
\State Compute: $\gradi{s} = \del f_i(\xth{s}{\N{i}})$
\State Quantize gradient and send to all $j \in \cN{i}$:
\Statex $~~~~~~~~~~~~~~~\qgradi{s} = \Qbi(\gradi{s}) = \gradi{s} + \bni{s}$
\State Compute: $\htil{s}_i = \frac{1}{N} \sum_{j \in \N{i}} \B_{ij} \qgrad{j}{s}$ 
\State Compute: $v^{(s)}_{i j} = - \B_{ij} \qgrad{j}{s} + \htil{s}_i$ for all $j \in \cN{i}$
\State Update quantizer parameters: 
\Statex $~~~~~~\Ucs{s} = C_c \kappa^{(s+1)/2}$,   $\cmid= \xth{s}{i}$, \\
\Statex $~~~~~~\Uds{s} =C_d \kappa^{(s+1)/2}$,   $\dmid = \qgradi{s}$
\State $\xs{0}_i = \xt{s}_i$
    \For{$t=0,1,\ldots,T-1$}
    	\State Randomly pick $\lt \in \{1,2,3, \ldots, N\}$
	\If{$i \in \N{\lt}$}
            	\State Quantize local variable and send to $\lt$: 
            	\Statex $~~~~~~~~~~~~~~~\xsh{t}_i = \Qci(\xs{t}_i) = \xs{t}_i + \cni{s_t}$
            	\If{$ i = \lt $}
		\State Compute: $\gradi{s_t} = \del f_i(\xsh{t}_{\N{i}})$
            	\State Quantize gradient and send to all $j \in \cN{i}$:
		\Statex $~~~~~~~~~~~~~~~\qgradi{s_t} = \Qdi(\gradi{s_t}) =  \gradi{s_t} + \dni{s_t}$
            	\EndIf
		\State Update local variable:  \\
		\Statex $~~~~~~~~~~~~~\xs{t+1}_i = \prox {\xs{t}_i - \eta(\B_{i \ell} \qgrad{\ell}{s_t}  + v_{i \ell}^{(s)})}$
	\Else
		\State Update local variable:  \\
		\Statex $~~~~~~~~~~~~~\xs{t+1}_i = \prox{\xs{t} - \eta \htil{s}_i}$
	\EndIf
	
    \EndFor
\State $\xt{s+1}_i =  \frac{1}{T} \sum_{t=1}^T \xs{t}$
\EndFor
\end{algorithmic}
\end{algorithm}

\section{Algorithm} \label{algorithm.sec}
Our distributed algorithm is given in Algorithm~\ref{SPGdist.alg}.  In each outer iteration $s$, node $i$ quantizes its iterate $\xt{s}_{i}$ and the gradient $\gradi{s}$ and sends it to all of its neighbors.  These values are quantized using two subtractively dithered quantizers, $\Qai$ and $\Qbi$,
whereby the sender (node $i$) sends an $n$ bit representation and the recipient reconstructs the value from this representation and subtracts the dither.
The midpoints for $\Qai$ and $\Qbi$ are set to be the quantized values from the previous iteration.  Thus, the recipients already know these midpoints.
The quantized values (after the dither is subtracted) are denoted by $\xth{s}{i}$ and $\qgradi{s}$, and the quantization errors are $\ani{s}$ and $\bni{s}$, respectively. 

For every iteration $s$ of the outer loop of the algorithm, there is an inner loop of $T$ iterations.  
In each inner iteration, a single node $\ell$, chosen at random, computes its gradient.  To do this, node $\ell$ and its neighbors exchange their states $\xs{t}_i$ 
and gradients $\gradi{s_t}$. These values are quantized using two subtractively dithered quantizers, $\Qci$ and $\Qdi$.
The midpoints for these quantizers are $\xth{s}{i}$ and $\qgrad{i}{s}$.  Each node sends these values to their neighbors before the inner loop, so 
all nodes are aware of the midpoints.  The quantized values (after the dither is subtracted) are denoted by  $\xsh{t}$ and $\qgradi{s_t}$,
and their quantization errors are $\cni{s_t}$ and $\dni{s_t}$,  respectively. The quantization interval bounds $\Uas{s}$, $\Ubs{s}$, $\Ucs{s}$, and $\Uds{s}$,
are initialized to $C_a$, $C_b$, $C_c$, and $C_d$, respectively, and each iteration, the bounds are multiplied by $\kappa^{1/2}$.
Thus the quantizers are \emph{refined} in each iteration.

The quantizers limit the length of a single variable transmission to $n$ bits. 
In the outer loop of the algorithm, each node $i$ sends its local variable, consisting of $m_i$ quantized components, to every neighbor. It also sends its gradient, consisting of $|\cN{i}| m_i$ quantized components to every neighbor.
Thus the number of bits exchanged by all nodes is $n \sum_{i=1}^N |\cN{i}| m_i + |\cN{i}|^2 m_i$ bits.
In each inner iteration, only nodes $j \in \cN{\ell}$ exchange messages. Each node $j$ quantizes $m_j$
state variables and sends them to  node $\ell$. This yields a transmission of $n\sum_{j \in \cN{\ell}} m_j$ bits in total. 
In turn, node $\ell$  quantizes its gradient and sends it to all of its neighbors, which is $n|\cN{\ell}|^2m_{\ell}$ total bits. Thus, in each inner iteration 
$n(|\cN{\ell}|^2m_{\ell} + \sum_{j \in \cN{\ell}} m_j)$ bits are transmitted.
The total number of bits transmitted in a single outer iteration is therefore, 
\[
n\left(\sum_{i=1}^N \left(|\cN{i}| m_i (1 + |\cN{i}|)\right) + \sum_{t=0}^{T-1}\left(|\cN{\ell}|^2m_{\ell} + \sum_{j \in \cN{\ell}} m_j\right)\right).
\]

 Let $D = \max_i |\cN{i}|$ and  $\overline{m} = \max_i m_i$. An upper bound on the number bits transmitted by the algorithm in each outer iteration is $n\overline{m}(N+T)(D+D^2)$.

%

\section{Algorithm Analysis}\label{analysis.sec}
We now present our analysis of Algorithm~\ref{SPGdist.alg}.
First we show that the algorithm is equivalent to Algorithm~\ref{centStochastic.alg}, where the quantization errors are
encapsulated in the error term $\ec{s_t}$.  We also give an explicit expression for this error term.
\begin{lemma}
Algorithm~\ref{SPGdist.alg} is equivalent to the Inexact Prox-SVG method in Algorithm~\ref{centStochastic.alg}, with 
\begin{align*} \label{qe.eq}
&\ec{s_t} = \cA{\ell}\tp \left( \nabla f_{\ell} (\xsh{t}_{\N{\ell}})- \nabla f_{\ell} (\xc{s_t}_{\N{\ell}}) \right) + \cA{\ell}\tp d_{\ell}^{(s_{t})}  \\
&~~-\cA{\ell}\tp \left( \nabla f_{\ell} (\xth{s}{\N{\ell}})- \nabla f_{\ell} (\xta{s}_{\N{\ell}}) \right) - \cA{\ell}\tp b_{\ell}^{(s)} \\
&~~+ \textstyle \frac{1}{N} \textstyle \sum_{i=1}^N \cA{i}\tp  \left( \nabla f_{i} (\xth{s}{\N{i}})- \nabla f_{i} (\xta{s}_{\N{i}}) \right) + \frac{1}{N}\sum_{i=1}^N  \cA{i}\tp \bni{s}.
\end{align*}
Further, $\ex{\| \ec{s_t} \|^2 }$ is upper-bounded by,
\begin{align*}
\ex \| \ec{s_t} \|^2 &\leq 2 \Lmax^2 \sum_{j \in \cN{\ell}} \ex \| c_j^{({s_t})} \|^2 + 2 \Lmax^2  \sum_{j \in \cN{\ell}} \ex \| a_j^{(s)} \|^2 \\
& + \ex \| d_{\ell}^{(s_{t})} \|^2 +2 \ex \|b_{\ell}^{(s)} \|^2 + \frac{2}{N^2} \sum_{i=1}^N \ex \| \bni{s} \|^2.
\end{align*}
\end{lemma}
\begin{proof}
The error $\ec{s_t}$ is:
\begin{align*}
&\ec{s_t} = \textstyle {\cA{\ell}}\tp  \qgrad{\ell}{s_t}  - \cA{\ell}\tp \qgrad{\ell}{s} + \frac{1}{N}\sum_{i=1}^N {\cA{i}}\tp \qgradi{s} \\
&~~~~~~- \left( \textstyle {\cA{\ell}}\tp \nabla f_{\ell} (\xc{s_t}_{\N{\ell}}) - {\cA{\ell}}\tp \nabla f_{\ell} (\xta{s}_{\N{\ell}} ) \right.\\
&~~~~~~~~~~~~~~~~~~~~~~~+\left.  \textstyle \frac{1}{N}\sum_{i=1}^N {\cA{i}}\tp  \nabla f_{i} (\xta{s}_{\N{i}})\right) \\
&= \cA{\ell}\tp \left( \nabla f_{\ell} (\xsh{t}_{\N{\ell}})- \nabla f_{\ell} (\xc{s_t}_{\N{\ell}}) \right) + \cA{\ell}\tp d_{\ell}^{(s_{t})}  \\
& -\cA{\ell}\tp \left( \nabla f_{\ell} (\xth{s}{\N{\ell}})- \nabla f_{\ell} (\xta{s}_{\N{\ell}}) \right) - \cA{\ell}\tp b_{\ell}^{(s)} \\
&+ \textstyle \frac{1}{N} \textstyle \sum_{i=1}^N \cA{i}\tp  \left( \nabla f_{i} (\xth{s}{\N{i}})- \nabla f_{i} (\xta{s}_{\N{i}}) \right) + \frac{1}{N}\sum_{i=1}^N  \cA{i}\tp \bni{s}.
 \end{align*}
We note that all quantization errors are zero-mean.  Further, by Assumption~\ref{je.assum}, $\expec{ \nabla f_i(x + \delta)} = \nabla f_i(x)$,
for a zero-mean random variable $\delta$.  Therefore, $\expec{\ec{s_t}} = 0$.  

We now show that $\ec{s_t}$ is  is uncorrelated with 
  $\xc{s_t}$ and the gradients $\nabla f_{\ell} (\xc{s_t}_{\N{\ell}})$, $\ell =1, \ldots, N$.
 Clearly,  $\xc{s_t}$ and $\nabla f_{\ell} (\xc{s_t}_{\N{\ell}})$ are uncorrelated with the terms of $\ec{s_t}$ 
 containing $d_{\ell}^{(s_{t})}$, $b_{\ell}^{(s)}$, and $\bni{s}$.
In accordance with Assumption~\ref{je.assum}, the gradients $\nabla f_{\ell}$ and $\nabla f_i$ are either linear or constant.
If they are constant, then 
$\nabla f_{\ell} (\xsh{t}_{\N{\ell}})- \nabla f_{\ell} (\xc{s_t}_{\N{\ell}}) = 0$ and  $\nabla f_{i} (\xth{s}{\N{i}})- \nabla f_{i} (\xta{s}_{\N{i}}) = 0$.
Thus, the terms in $\ec{s_t}$ containing these differences are also 0.
If they are linear, e.g., $\nabla f_{\ell}(z) = Hz + h$, for an appropriately sized, matrix $H$ and vector $h$ (possibly 0).
Then,
\begin{align*}
&\nabla f_{\ell} (\xsh{t}_{\N{\ell}})- \nabla f_{\ell} (\xc{s_t}_{\N{\ell}}) \\
&~~~~ = (H (\xc{s_t}_{\N{\ell}} + \cni{s_t}) + h) - (H\xc{s_t} + h)  = H \cni{s_t}.
\end{align*}
 By Theorem~\ref{quantizer.thm}, $\cni{s_t}$ is uncorrelated with $\xc{s_t}$.  
It is clearly also uncorrelated with $\nabla f_{\ell}(\xc{s_t}_{\N{\ell}})$.  Similar arguments can be used to
show that $\xc{s_t}$ and $\nabla f_{\ell}(\xc{s_t}_{\N{\ell}})$ are uncorrelated with the remaining terms in $\ec{s_t}$.

 With respect to $\ex \| \ec{s_t} \|^2$, we have
 \begin{align*}
\ex \| \ec{s_t} \|^2 &= \ex \| \cA{\ell}\tp \left( \nabla f_{\ell} (\xsh{t}_{\N{\ell}})- \nabla f_{\ell} (\xc{s_t}_{\N{\ell}}) \right)  \\
&~~-\cA{\ell}\tp \left( \nabla f_{\ell} (\xth{s}{\N{\ell}})- \nabla f_{\ell} (\xta{s}_{\N{\ell}}) \right)   \\
&~~+ \textstyle \frac{1}{N} \textstyle \sum_{i=1}^N \cA{i}\tp  \left( \nabla f_{i} (\xth{s}{\N{i}})- \nabla f_{i} (\xta{s}_{\N{i}}) \right) \|^2 \\
& + \ex \| \cA{\ell}\tp d_{\ell}^{(s_{t})}  + \textstyle \frac{1}{N}\sum_{i=1}^N  \cA{i}\tp \bni{s} - \cA{\ell}\tp b_{\ell}^{(s)} \|^2.
 \end{align*}
The first term on the right hand side can be bounded using the fact that $\| a + b \|^2 \leq 2 \|a \|^2 + 2 \|b \|^2$, as
\begin{eqnarray*}
&&\leq 2 \ex \| \cA{\ell}\tp \left( \nabla f_{\ell} (\xsh{t}_{\N{\ell}})- \nabla f_{\ell} (\xc{s_t}_{\N{\ell}}) \right) \|^2 \\
&& ~~~~~~+ 2 \ex \|\cA{\ell}\tp \left( \nabla f_{\ell} (\xth{s}{\N{\ell}})- \nabla f_{\ell} (\xta{s}_{\N{\ell}}) \right)   \\
&&~~~~~~~~~~~~+ \textstyle \frac{1}{N} \textstyle \sum_{i=1}^N \cA{i}\tp  \left( \nabla f_{i} (\xth{s}{\N{i}})- \nabla f_{i} (\xta{s}_{\N{i}}) \right) \|^2. \\
\end{eqnarray*}
We now bound the first term in this expression, 
\begin{align*}
&2 \ex \| \cA{\ell}\tp \left( \nabla f_{\ell} (\xsh{t}_{\N{\ell}})- \nabla f_{\ell} (\xc{s_t}_{\N{\ell}}) \right) \|^2  \\
&~~~~~~\leq 2 \ex( L_i^2 \| \xsh{t}_{\N{\ell}} - \xc{s_t}_{\N{\ell}} \|^2) \leq 2 \Lmax^2 \sum_{j \in \cN{\ell}} \ex \| c_j^{({s_t})} \|^2,
\end{align*}
where the first inequality follows from Assumptions~\ref{fi.assum} and \ref{je.assum}
and the fact that $\| \cA{\ell} \| = 1$.  The second inequality follows from the independence of quantization errors (Theorem~\ref{quantizer.thm}).
Next we bound the second term, 
\begin{align*}
&2 \ex \|\cA{\ell}\tp \left( \nabla f_{\ell} (\xth{s}{\N{\ell}})- \nabla f_{\ell} (\xta{s}_{\N{\ell}}) \right)   \\
&~~+ \textstyle \frac{1}{N} \textstyle \sum_{i=1}^N \cA{i}\tp  \left( \nabla f_{i} (,\xth{s}{\N{i}})- \nabla f_{i} (\xta{s}_{\N{i}}) \right) \|^2\\
&= 2 \ex \|\cA{\ell}\tp \left( \nabla f_{\ell} (\xth{s}{\N{\ell}})- \nabla f_{\ell} (\xta{s}_{\N{\ell}}) \right) \\
&~~~~  - \expec{\cA{\ell}\tp \left( \nabla f_{\ell} (\xth{s}{\N{\ell}})- \nabla f_{\ell} (\xta{s}_{\N{\ell}}) \right)} \|^2 \\
& \leq 2 \ex \|\cA{\ell}\tp \left( \nabla f_{\ell} (\xth{s}{\N{\ell}})- \nabla f_{\ell} (\xta{s}_{\N{\ell}}) \right) \|^2 \\
& \textstyle  \leq 2 \ex( L_i^2 \| (\xth{s}{\N{\ell}} - \xta{s}_{\N{\ell}} \|^2) \\
& \textstyle \leq 2 \Lmax^2  \sum_{j \in \cN{\ell}} \ex \| a_j^{(s)} \|^2,
\end{align*}
where the first inequality uses the fact that for a random variable $\upsilon$, $ \ex \| \upsilon- \ex\upsilon \|^2 = \ex \| \upsilon \|^2 - \| \ex \upsilon \|^2 \leq \ex \| \upsilon \|^2$.
The remaining inequalities follow from Assumptions~\ref{fi.assum} and \ref{je.assum}, the fact that $\| \cA{\ell} \| = 1$, and the independence 
of the quantization errors.

Finally, again from the independence of the quantization errors, we have, 
\begin{align*}
& \ex \| \cA{\ell}\tp d_{\ell}^{(s_{t})}  + \textstyle \frac{1}{N}\sum_{i=1}^N  \cA{i}\tp \bni{s} - \cA{\ell}\tp b_{\ell}^{(s)} \|^2 \\
& \textstyle  \leq \ex \| \cA{\ell}\tp d_{\ell}^{(s_{t})} \|^2 + \ex \| \frac{1}{N} \sum_{i=1}^N  \cA{i}\tp \bni{s} \|^2 - \cA{\ell}\tp b_{\ell}^{(s)} \|^2 \\
&\textstyle  \leq \ex \| d_{\ell}^{(s_{t})} \|^2 +2 \ex \|b_{\ell}^{(s)} \|^2 + \frac{2}{N^2} \sum_{i=1}^N \ex \| \bni{s} \|^2.
\end{align*}
Combining these bounds, we obtain the desired result, 
\begin{align*}
\ex \| \ec{s_t} \|^2 &\leq 2 \Lmax^2 \sum_{j \in \cN{\ell}} \ex \| c_j^{({s_t})} \|^2 + 2 \Lmax^2  \sum_{j \in \cN{\ell}} \ex \| a_j^{(s)} \|^2 \\
& + \ex \| d_{\ell}^{(s_{t})} \|^2 +2 \ex \|b_{\ell}^{(s)} \|^2 + \frac{2}{N^2} \sum_{i=1}^N \ex \| \bni{s} \|^2.
\end{align*}
\end{proof}

We next show that, if all of the values fall within their respective quantization intervals, 
then the error term $\Gamma^{(s)}$  decreases linearly with rate $\kappa$, and thus the algorithm converges to the optimal solution linearly with rate $\kappa$.
\begin{theorem} \label{quantconverge.thm}
Given $p$, if for all $1 \leq s \leq (p-1)$, 
the values  of $\xt{s}_i$, $\gradi{s}$, $\xs{t}$, and $\gradi{s_t}$
fall inside of the respective quantization intervals $\Qai$, $\Qbi$, $\Qci$, and $\Qdi$, 
then $\Gamma^{(k)} \leq C \kappa^{k}$, where, 
\[
C = \frac{D T \overline{m}}{12 (2^{\ell} - 1)^2} \left( 2 \Lmax^2 (C_a + C_b) + 2 \textstyle (\frac{N+1}{N}) C_b + C_d\right),
\]
with $D = \max_i | \cN{i} |$ and $\overline{m} = \max_i m_i$.

It follows that, for $\alpha < \kappa < 1$, 
\begin{align*}
&\expec{G(\xta{s}) - G(\xopt)} \\
&~~~~~ \leq \kappa^s \left(G(\xt{0}) - G(x^\star) + \beta C \left( \frac{1}{1 - \frac{\alpha}{\kappa}} \right) \right)
\end{align*}
\end{theorem}
\begin{proof}
First we note that, by Theorem~\ref{quantizer.thm} and the update rule for the quantization intervals, we have:
\begin{align*}
\ex{\|\ani{s} \|^2} & \leq  \textstyle \frac{\overline{m}}{12} \left(\frac{ \Uas{s}}{2^\ell  -1}\right)^2   \leq  \frac{\overline{m}}{12(2^\ell  -1)^2} C_a \kappa^s  \\
\ex{\|\bni{s} \|^2} & \leq \textstyle \frac{D \overline{m}}{12} \left(\frac{ \Ubs{s}}{2^\ell  -1}\right)^2 \leq  \frac{D\overline{m}}{12(2^\ell  -1)^2} C_b \kappa^s  \\
\ex{\|\cni{s_t} \|^2}  & \leq \textstyle \frac{\overline{m}}{12}\left(\frac{ \Ucs{s}}{2^\ell  -1}\right)^2     \leq \frac{\overline{m}}{12(2^\ell  -1)^2} C_c \kappa^s \\
\ex{\|\dni{s_t} \|^2} & \leq \textstyle \frac{D \overline{m}}{12} \left(\frac{ \Uds{s}}{2^\ell  -1}\right)^2 \leq \frac{D\overline{m}}{12(2^\ell  -1)^2} C_d \kappa^s .
\end{align*}
We use these inequalities to bound $\| \ec{s_t} \|^2$,
\begin{align*}
&\ex \| \ec{s_t} \|^2 \leq 2 \textstyle \Lmax^2 D \left( \frac{\overline{m}}{12(2^\ell  -1)^2} C_c \kappa^s \right) \\
& + \textstyle2  \left( \Lmax^2 D \frac{\overline{m}}{12(2^\ell  -1)^2} C_a \kappa^s \right)  + \frac{D\overline{m}}{12(2^\ell  -1)^2} C_d \kappa^s \\
& + \textstyle2 \left(\frac{D\overline{m}}{12(2^\ell  -1)^2} C_b \kappa^s\right)  + \frac{2}{N} \left( \frac{D\overline{m}}{12(2^\ell  -1)^2} C_b \kappa^s\right) \\
&= \textstyle\frac{D \overline{m}}{12 (2^{\ell} - 1)^2} \left( 2 \Lmax^2 (C_a + C_c) + 2 \textstyle (\frac{N+1}{N}) C_b + C_d\right) \kappa^s.
\end{align*}
Summing over $t=0, \ldots, T-1$, we obtain,
\[
\Gamma^{(s)} = \sum_{t=0}^{T-1} \ex \| \ec{s_t} \|^2 \leq C \kappa^s,
\]
where 
\[
C = \textstyle\frac{D T \overline{m}}{12 (2^{\ell} - 1)^2} \left( 2 \Lmax^2 (C_a + C_c) + 2 \textstyle (\frac{N+1}{N}) C_b + C_d\right)
\]
Applying Theorem~\ref{centralized.thm}, with $\kappa > \alpha$, we have 
\begin{align*}
&\expec{G(\xta{s}) - G(\xopt)} \\
&~~~~~ \leq \alpha^s \left(G(\xt{0}) - G(\xopt)  \right) + \beta \sum^{s}_{i=1} \alpha^{s-i} C \kappa^i \\
&~~~~~ \leq \kappa^s \left(G(\xt{0}) - G(\xopt) + C \beta \sum^{s}_{i=1} \kappa^{-(s-i)} \alpha^{s-i} \right) \\
&~~~~~ \leq \textstyle\kappa^s \left(G(\xt{0}) - G(\xopt) + C \beta \frac{1 - (\frac{\alpha}{\kappa})^s}{1 - \frac{\alpha}{\kappa} } \right) \\
&~~~~~ \leq \textstyle\kappa^s \left(G(\xt{0}) - G(\xopt) + C \beta \left( \frac{1}{1 - \frac{\alpha}{\kappa} }\right) \right).
\end{align*}
\end{proof}

While we do not yet have theoretical guarantees that all values will fall within their quantization intervals,
our simulations indicate that is always possible to find parameters $C_a$, $C_b$, $C_c$, and $C_d$, for which all values lie within their quantization intervals for all iterations.
Thus, in practice, our algorithm achieves a linear convergence rate.
We anticipate that it is possible to develop a programmatic approach, similar to that in~\cite{PY2015a}, to identify values for $C_a$, $C_b$, $C_c$, and $C_d$ that guarantee linear convergence.  
This is a subject of current work.

\section{Numerical Example} \label{results.sec}
This section illustrates the performance of Algorithm \ref{SPGdist.alg} by solving a distributed linear regression problem with elastic net regularization.

We randomly generate a $d$-regular graph with $N = 40$ and uniform degree of 8, i.e., $\forall i~|\cN{i}| = 9$.
We set each subsystem size, $m_i$, to be 10.   Each node has a local function $f_i(x_{\cN{i}}) = \norm{H_i x_{\cN{i}} - h_i}^2$ where $H_i$ is a $80 \times 90$ random matrix.
We generate $h_i$ by first generating a random vector $x$ and then computing $h_i = H_i x$.
The global objective function is:
\[
G(x) = \frac{1}{N}\sum_i^N f_i(x_{\cN{i}}) + \lambda_1 \| x \|_2 + \frac{\lambda_2}{2} \| x \|_1.
\]
 
This simulation was implemented in Matlab and the optimal value $x^{\star}$ was computed using CVX. 
We set the total number of inner iterations to be  $T = 2  N$ and use the step size $\eta = 0.1 / \Lmax$.  With these values, $\alpha < 1$, as required by Theorem~\ref{centralized.thm}.
We set $\kappa = 0.97$, which ensures that $\kappa > \alpha$. 
We use the  quantization parameters $C_a = 50, ~C_b = 300, ~C_c = 50, ~C_d = 400$.   With these parameters, the algorithms values always fell within their quantization intervals.

\begin{figure}[t]
\centering
\includegraphics[scale=.5]{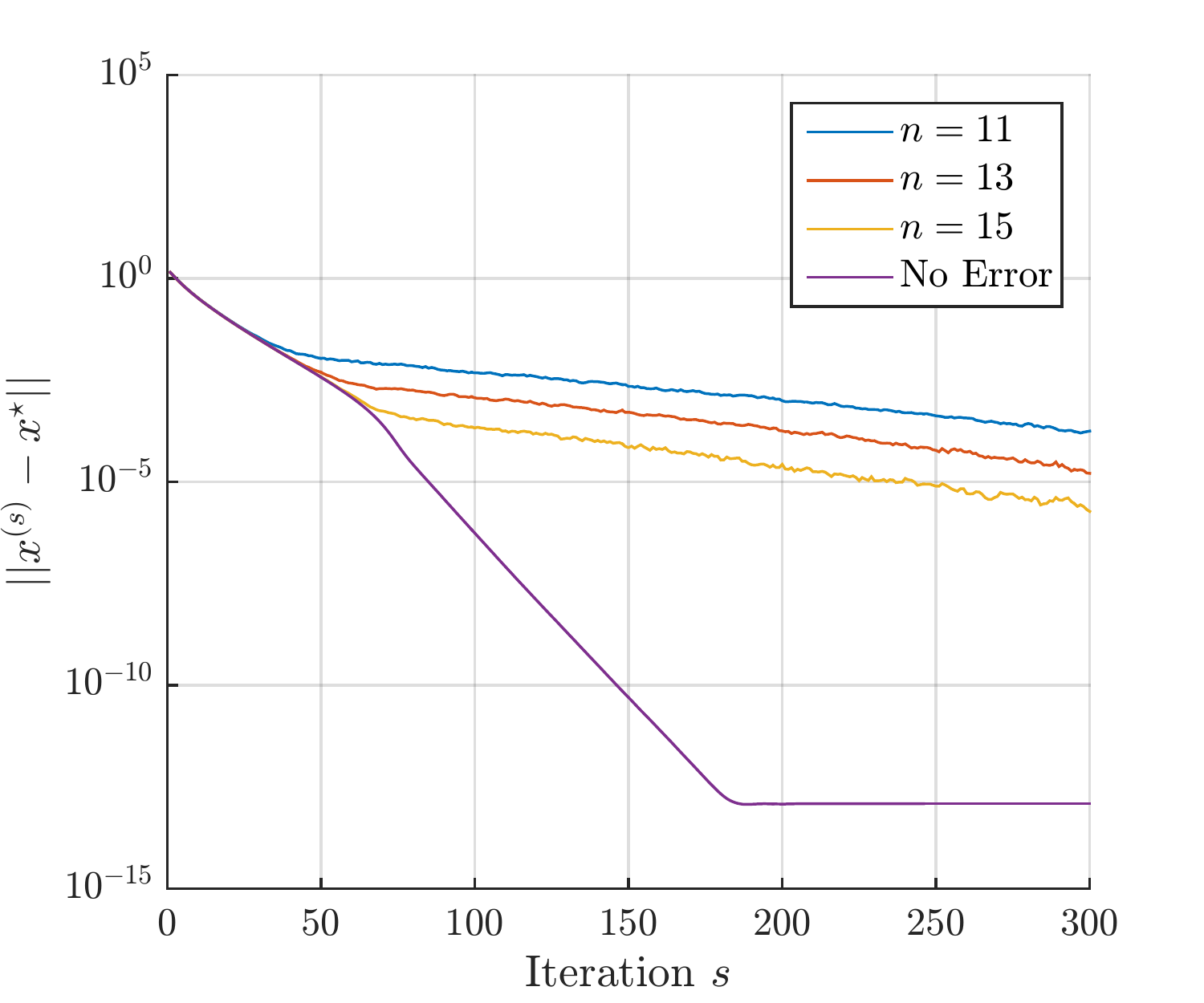}
\caption{Comparison of the performance of Algorithm \ref{SPGdist.alg} with differing quantized message lengths and that with no quantization applied.}  \label{errCompare.fig}
\end{figure}

 Fig.~\ref{errCompare.fig} shows the performance of the algorithm where the number of bits  $n$ is 11, 13, and 15, as well as the performance of the algorithm without quantization.  In these results, $x^{(s)}$ is the concatenation of the $\xt{s}_{i}$ vectors, for $i=1 \ldots N$.  
 It is important to note the rate of convergence of the algorithm in all four cases is linear, and, performance improves as the number of bits increases.

\section{Conclusion} \label{conclusion.sec}
We have presented a distributed algorithm for regularized regression in 
communication-constrained networks.
This algorithm is based on recently proposed semi-stochastic proximal gradient methods.  Our algorithm reduces communication requirements by
(1) using a stochastic approach where only a subset of nodes communicate in each iteration and (2) quantizing all messages.
We have shown that this distributed algorithm is equivalent to a centralized version with inexact gradient computations, and we have used this equivalence
to analyze the convergence rate of the distributed method.
Finally, we have demonstrated the performance of our algorithm in numerical simulations.

In future work, we plan to extend our theoretical analysis to develop a programmatic way to identify initial quantization intervals. We also plan to explore 
the integration of more complex regularization functions.

\appendix

\subsection*{Proof of Theorem~\ref{centralized.thm}}\label{thm1.app}

We first restate  some useful results from~\cite{xiao2014proximal}.
\begin{lemma} \label{useful1.lem}
Let 
\[
\wc{s_t}= \nabla f_{\ell}(\xc{s_{t-1}}) - \nabla f_{\ell}(\xta{s}) + \nabla F(\xta{s}).  
\]
Then, conditioned on $\xc{s_{t-1}}$, ${\expec{ \wc{s_t}} = \nabla F(\xc{s_{t-1}})}$ and 
\begin{align*}
&\ex \| \wc{s_t} - \nabla F(\xc{s_{t-1}}) \|^2 \\
&~~~~~ \leq 4 L (G(\xc{s_{t-1}}) - G(\xopt) + G(\xta{s}) - G(\xopt)),
\end{align*}
\end{lemma}

\begin{lemma} \label{useful2.lem}
Let $G(x) = F(x) + R(x)$, where $G$ is strongly convex, $\del F(x)$ is Lipschitz continuous with parameter $L$.  Further let $F(x)$ and $R(x)$ have convexity parameters $\mu_F$ and $\mu_R$, respectively.  In other words, if $F(x)$ ($R(x)$) is strongly convex, then
$\mu_F$ ($\mu_R$) is its strong convexity parameter; if $F(x)$ ($R(x)$) is only convex, then $\mu_F$ ($\mu_R$) is 0.
For any $x \in \text{dom}(R)$ and any $v \in \mR^P$, define,
\begin{align*}
\xplus &= \prox{x - \eta v} \\
h &= \textstyle \frac{1}{\eta}(x - \xplus) \\
\Delta &= v - \nabla F(x),
\end{align*}
where $0 < \eta < \frac{1}{L}$.  Then, for any $y \in \mR^{P}$,
\begin{align*}
G(y) &\geq  \textstyle G(\xplus) + h\tp (y-x) + \frac{\eta}{2} \|h \|^2 + \frac{\mu_F}{2} \|y - x \|^2 \\
&~~~~~ + \textstyle \frac{\mu_R}{2} \| y - \xplus \|^2 + \Delta\tp (\xplus - y).
\end{align*}
\end{lemma}

We now proceed to prove Theorem~\ref{centralized.thm}.  For brevity, we omit some details that are identical to those in the 
proof of Theorem 3.1 in \cite{xiao2014proximal}.  We have indicated these omissions below.

\begin{proof}
First, we define 
\begin{align*}
\hc{s_t} &= \textstyle \frac{1}{\eta}(\xc{s_{t-1}} - \xc{s_{t}}) \\
&= \textstyle \frac{1}{\eta}(\xc{s_{t-1}} - \prox{\xc{s_{t-1}} - \eta \vc{s_{t-1}}}),
\end{align*}
where $\vc{s_{t-1}}$ is as defined in Algorithm~\ref{centStochastic.alg}.

We analyze the change in the distance between $\xc{s_t}$ and $\xopt$ in a single inner iteration,
\begin{align*}
& \| \xc{s_t} - \xopt \|^2 = \| \xc{s_{t-1}} - \eta \hc{s_t} - \xopt \|^2 \\
&=\| \xc{s_{t-1}} \xopt \|^2  - 2 \eta {\hc{s_t}} \tp(\xc{s_{t-1}} - \xopt) + \eta^2 \| \hc{s_t} \|^2.
\end{align*}
We next apply Lemma~\ref{useful2.lem}, with $x = \xc{s_{t-1}}$, $\xplus = \xc{s_t}$, $h = \hc{s_t}$, $v = \vc{s_{t-1}}$, and $y= \xopt$, to obtain,
\begin{align*}
&- \textstyle {\hc{s_t}} \tp(\xc{s_{t-1}} - \xopt) + \frac{\eta}{2} \| \hc{s_t} \|^2  \\
&~~~~~~~~~~~~\leq \textstyle G(\xopt) - G(\xc{s_t}) - \frac{\mu_F}{2} \| \xc{s_{t-1}} - \xopt \|^2   \\
&~~~~~~~~~~~~~~~-\textstyle \frac{\mu_R}{2} \| \xc{s_t} - \xopt \|^2 - {\Delta^{(s_{t})}} \tp (\xc{s_t} - \xopt),
\end{align*}
where $\Delta^{(s_{t})} = \vc{s_{t-1}} - \del F(\xc{s_{t-1}}) = \wc{s_{t-1}} + \ec{s_{t-1}} - \del F(\xc{s_{t-1}})$.
This implies,
\begin{align*}
\| \xc{s_t} - \xopt \|^2 &\leq \| \xc{s_{t-1}} - \xopt \|^2 - 2 \eta ( G( \xc{s_t}) - G(\xopt)) \\
&~~~~~- 2 \eta {\Delta^{(s_{t})}} \tp (\xc{s_t} - \xopt).
\end{align*}

We follow the same reasoning as in the proof of Theorem 3.1 in \cite{xiao2014proximal} to obtain the following expression,
which is conditioned on $\xc{s_{t-1}}$ and takes expectation with respect to $\ell$,
\begin{align*}
&\ex \| \xc{s_t} - \xopt \|^2 \leq \| \xc{s_{t-1}} - \xopt \|^2  \\
&~~~- 2 \eta \ex({G(\xc{s_t}) - G(\xopt)}) + 2 \eta^2 \ex \|\Delta^{(s_{t})} \|^2 \\
&~~~- 2 \eta \expec{ {\Delta^{(s_{t})}} \tp (\overline{x}^{(s_t)} - \xopt)},
\end{align*}
where 
\[
\overline{x}^{(s_t)} = \prox{\xc{s_{t-1}} - \eta \nabla F (\xc{s_{t-1}})}.
\]

Since  $\ell$ and $\ec{s_{t-1}}$ are independent of $\overline{x}^{(s_t)}$ and $\xopt$, and since $\ec{s_{t-1}}$ is 
zero-mean,
\[
\expec{ {\Delta^{(s_{t})}} \tp (\overline{x}^{(s_t)} - \xopt)} = (\ex {\Delta^{(s_{t})}})\tp (\overline{x}^{(s_t)} - \xopt) = 0.
\]
Further, since  $\ec{s_{t-1}}$ is independent of $\wc{s_{t-1}}$ and $\nabla F(\xc{s_{t-1}})$, 
\[
\ex \|\Delta^{(s_{t})} \|^2  = \ex \| \wc{s_{t-1}} - \nabla F(\xc{s_{t-1}}) \|^2 + \ex \| \ec{s_{t-1}} \|^2
\]
Applying Lemma~\ref{useful1.lem}, we  obtain,
\begin{align*}
&\ex \| \xc{s_t} - \xopt \|^2 \leq \| \xc{s_{t-1}} - \xopt \|^2 - 2 \eta \ex(G(\xc{s_t}) - G(\xopt)) \\
&~~+  8 \Lmax \eta^2 (G(\xc{s_{t-1}}) - G(\xopt) + G(\xta{s}) - G(\xopt)) \\
&~~~+ 2\eta^2 \ex \| \ec{s_{t-1}} \|^2
\end{align*}

We consider a single execution of the inner iteration of the algorithm, so $\xc{s_0} = \xta{s}$ and $\xta{s+1} = \frac{1}{T}\sum_{t=1}^T \xc{s_{t}}$.
Summing over $t=1, \ldots, T$ on both sides gives and taking expectation over $\ell$, for $t=1, \ldots, T$ gives us,
\begin{align*}
& \ex \| \xc{s_T} - \xopt \|^2 + 2 \eta \ex({G(\xc{s_T}) - G(\xopt)})  \\
& + 2 \eta(1 - 4 \Lmax \eta) \sum_{t=1}^{T-1} \ex({G(\xc{s_t}) - G(\xopt)})\\
&~~~~~~ \leq  \|\xc{s_0} - \xopt \|^2 + 8 \Lmax \eta^2(G(\xc{s_0}) - G(\xopt) \\
&~~~~~~~~~+ T(G(\xt{s}) - G(\xopt)) + 2 \eta^2 \Gamma^{(s)}.
\end{align*}

Following the same reasoning as in~\cite{xiao2014proximal}, we obtain,
\begin{align*}
\ex({G(\xta{s+1}) - G(\xopt)}) \leq \alpha \expec{G(\xta{s} - G(\xopt)} + \beta \Gamma^{(s)}. 
\end{align*}
Applying this bound recursively, we obtain the expression in our theorem.
\end{proof}

\bibliographystyle{IEEETran}
\bibliography{quantize,sep}

\end{document}